\documentclass[12pt]{amsart}

\usepackage{cmap}
\usepackage[T1]{fontenc}
\usepackage{lmodern}
\usepackage{amssymb}
\usepackage[ruled,vlined]{algorithm2e}
\usepackage[letterpaper, left=2.5cm, right=2.5cm, top=2.5cm,
bottom=2.5cm]{geometry}
\usepackage{amsmath}
\usepackage[english]{babel}
\newtheorem{theorem}{Theorem}
\newtheorem{case}{Case}
\newtheorem{conjecture}{Conjecture}
\newtheorem{corollary}{Corollary}
\newtheorem{lemma}{Lemma}
\newtheorem{problem}{Problem}
\numberwithin{equation}{section}

\usepackage{xcolor}

\usepackage{hyperref}
\hypersetup{
	colorlinks,
	linkcolor={blue!60!black},
	citecolor={blue!60!black},
	urlcolor={blue!60!black}
}

\begin{document}
	\title{Epistemic fair division of independence structures}
	
		\author{Marcin Anholcer}
	\address{Institute of Informatics and Quantitative Economics, Poznań University of Economics and Business, Poznań, Poland}
	\email{marcin.anholcer@ue.poznan.pl}
	
		\author{Maciej Bartkowiak}
	\address{Doctoral School, Poznań University of Economics and Business, Poznań, Poland}
	\email{maciej.bartkowiak@ue.poznan.pl}
	
		\author{Bartłomiej Bosek}
	\address{Institute of Theoretical Computer Science, Faculty of Mathematics and Computer Science, Jagiellonian University, Kraków, Poland}
	\email{bartlomiej.bosek@uj.edu.pl}
	
		\author{Jarosław Grytczuk}
\address{Faculty of Mathematics and Information Science, Warsaw University
	of Technology, 00-662 Warsaw, Poland}
\email{jaroslaw.grytczuk@pw.edu.pl}

\thanks{This research was funded by the National Science Centre, Poland (grant number: 2023/51/B/HS4/00829).}

\begin{abstract}
	We study the problem of fair division of indivisible goods with constraints imposed by a prescribed \emph{independence structure}, that is, a family of subsets of goods closed under taking subsets. As a motivating example, imagine that the goods to be divided are the available connections in a logistic, financial, or social network. The admissible bundle of goods for each agent must correspond to an \emph{acyclic} set of edges, corresponding to a basic feasible solution to a linear network problem to be solved (like, e.g., the Minimum Cost Flow Problem). Suppose that all agents assign the same value to each good (in the example, the network connections are equally important for every agent) and evaluate each bundle by summing the values of its goods. Is there a \emph{fair partition} of the goods into such acyclic bundles?
	
	Surprisingly, the answer is yes, provided that the number of agents is at least the \emph{arboricity} of $G$ (the least number of parts in an acyclic partition of the edges of $G$), and the fairness requirement is envy-freeness up to one good (EF1). The situation becomes more mysterious when agents have arbitrary additive valuations. Our main result guarantees that, in this case, \emph{epistemic} EF1 partitions always exist, which means that each agent receives an acyclic bundle for which there exists a feasible partition of the remaining goods into acyclic bundles that they do not envy up to one good.
	
	We derive this conclusion from a general result for abstract \emph{independence structures} defined on the sets of goods. We also discuss connections with several conjectures concerning \emph{matroids}. In particular, we prove that any \emph{Hamiltonian matroid} partitionable into two independent sets admits an EF1 bipartition with respect to a common monotone valuation.

    We complement our results with a constructive perspective: we present explicitly two algorithms for computing the fair allocations described above. Finally, we provide illustrative examples to demonstrate these algorithms on specific instances.
\end{abstract}
\maketitle

\section{Introduction}
Fair division studies how to allocate resources among several agents in a fair manner. Its modern formal treatment is often traced back to Steinhaus' classical work on cake cutting \cite{steinhaus1948}, and has since developed into two major strands: the division of divisible resources, surveyed for example by Brams and Taylor \cite{Brams1996} and Moulin \cite{Hervebook2003}, and the division of indivisible goods, which has become a central topic in computational social choice and algorithmic game theory; see, e.g., \cite{Amanatidis2022,Procaccia2020}. Closely related questions also arise in economics and operations research, where indivisible resources must often be allocated subject to institutional or technological restrictions. Among the examples, one can list dividing the marital property after divorce, matching projects with employees \cite{Brams01122015}, dividing goods by food banks \cite{Aleksandrov20152540}, allocation of courses to students \cite{Biswas20242162}, recommendation algorithms allocating products to re-sellers in social commerce platforms \cite{Gupta20233744} and others.

In the indivisible setting, exact envy-freeness may fail to exist, which has motivated the study of relaxations. Two of the most prominent are envy-freeness up to one good (EF1) and envy-freeness up to any good (EFX). EF1 was already central in the early algorithmic literature on approximate envy-freeness \cite{Lipton2004} and is now regarded as one of the standard benchmark notions for indivisible goods; see also \cite{Budish2011}. EFX, introduced and systematically studied by Caragiannis et al.\ and Plaut and Roughgarden, is a stronger and conceptually very appealing relaxation, but its existence for general instances, even with additive valuations, remains one of the major open problems in fair division \cite{Caragiannis2019,Plaut2020}. Recently, Akrami et al. \cite{Akrami2026EFXCounterexample} presented a counterexample to the existence of EFX division with arbitrary monotone valuations.

Most of the classical EF1/EFX literature considers the unconstrained setting, where every subset of goods can in principle be assigned to an agent. In many applications, however, allocations must satisfy additional feasibility constraints. These may arise from capacities, categories, compatibility restrictions, or structural conditions on acceptable bundles \cite{DrorFeldmanSegalHalevi2023, SuksompongSurvey2021}. Such constraints capture, for instance, allocation rules arising from scheduling and matching, limited resource compatibility, or budget restrictions.

A natural and flexible framework for modelling such restrictions is provided by matroids. Matroid constraints subsume, among others, cardinality-type and partition-type restrictions, and in the graphical case they encode acyclicity through the graphic matroid. Thus, matroids offer a common language for expressing feasible bundles while preserving enough combinatorial structure to support existence and algorithmic analysis.

Fair division under bundle constraints has attracted growing attention in recent years. Biswas and Barman \cite{BarmanBiswas2018} studied fair division under cardinality constraints, establishing EF1 guarantees in that model and showing, in particular, that under identical additive valuations EF1 can also be computed efficiently under laminar matroid constraints. More generally, Dror, Feldman, and Segal-Halevi \cite{DrorFeldmanSegalHalevi2023} investigated fair division under heterogeneous matroid constraints, obtaining both positive and negative results for EF1 across several valuation and matroid classes. Recent work has also examined the interplay of fairness and efficiency under constraints, showing, for example, that maximum Nash welfare continues to yield meaningful approximate-EF1 guarantees under broad matroidal feasibility assumptions \cite{CooksonEbadianShah2025,WangChenNong2024}. Together, these results suggest that matroids form one of the most promising general frameworks for constrained fair division.

At the same time, graph-based models have opened several new directions in fair division. One line places a graph on the agents and studies fairness relative to local comparison or limited information. Aziz et al.\ \cite{Aziz2018socgraph} introduced epistemic notions of envy-freeness parameterized by a social graph, thereby formalizing the idea that an agent may observe only part of the allocation. More recently, Caragiannis et al.\ \cite{CaragiannisEtAl2023EEFX} proposed epistemic EFX (EEFX), a relaxation of EFX that always exists for additive valuations, and Akrami and Rathi \cite{AkramiRathi2025} extended this existence result to general monotone valuations. These notions are particularly appealing in settings where strong ex post fairness appears difficult to guarantee but one may still hope for fairness after a suitable epistemic reinterpretation of the allocation.

A different graph-based direction concerns structure on the goods themselves. For example, connected-bundle models require each agent's bundle to induce a connected subgraph in a given item graph \cite{Bei_2022,BILO2022}. More recently, Christodoulou et al.\ \cite{ChristodoulouEtAl2023} studied a valuation model in which agents are vertices, goods are edges, and each good is valuable only to its endpoints; they proved that EFX allocations always exist in this setting, even though deciding the existence of an EFX orientation is NP-hard. Subsequent work by Zeng and Mehta \cite{ZengMehta2025} and others refined the structural picture of EFX on graphs. These results highlight that graph structure can sometimes recover strong fairness guarantees unavailable in general.

In this paper, we focus on fair division of indivisible goods under independence-structure constraints, with matroids as the central motivating class. Our motivation comes in part from the graphic matroid viewpoint. If the goods are edges of a graph and feasible bundles are required to be independent in the graphic matroid, then every agent must receive a forest. This turns fair allocation into a decomposition problem: can one partition a weighted graph into feasible forests that are fair in the EF1 or EFX sense? More generally, given a common matroid on the ground set of goods, one may ask whether the goods admit a partition into independent bundles satisfying strong fairness guarantees. Such questions combine the global comparison aspect of EF1/EFX with the local exchange structure of matroids, and appear to be largely unexplored beyond several important special cases.

Our contribution is to study this constrained setting through both standard and epistemic fairness lenses. Conceptually, our results contribute to an emerging picture in which matroid structure provides a natural combinatorial framework for constrained fair division, while epistemic relaxations offer a promising route when stronger notions such as EFX are difficult or impossible to obtain directly.

Our main contributions are as follows. Theorem \ref{Theorem Epistemic Matroid} proves the existence of epistemic EF1 partitions for arbitrary additive valuations in every hereditary class of independence structures having the EF1 property. Theorem \ref{Theorem Matroid Hamiltonian} establishes a common-valuation EF1 bipartition for Hamiltonian matroids with chromatic number at most $2$. In addition, we present two algorithms that compute the corresponding allocations and demonstrate their behavior on illustrative examples, highlighting both the intuition and the key structural properties behind our constructions.

\section{The setting and the results}

We consider a variant of the fair division problem of indivisible goods in the following setting. Let $E=\{g_1,g_2,\ldots,g_m\}$ be the set of \emph{goods} and let $A=\{1,2,\ldots,n\}$ be the set of \emph{agents}. Each agent $i\in A$ assigns a real number $v_i(g_j)\geqslant 0$ to each good $g_j$. In this way, any subset of goods $X\subseteq E$, called a \emph{bundle}, gets a value $v_i(X)=\sum_{g_j\in X}v_i(g_j)$ given by agent $i$. Such valuations are called \emph{additive}. Notice that all the information about additive valuations can be stored in an $n\times m$ matrix $B=(b_{ij})$, where $b_{ij}=v_i(g_j)$. More generally, a valuation $v\colon 2^E\to \mathbb{R}_{\geqslant 0}$ is \emph{monotone} if $v(X)\leqslant v(Y)$ whenever $X\subseteq Y$. Every nonnegative additive valuation is monotone. Theorem \ref{Theorem Epistemic Matroid} concerns additive valuations, whereas Lemma \ref{Lemma Circle Equal}, Theorem \ref{Theorem Matroid Hamiltonian}, and the conjectures in Section \ref{section:finalremarks} are stated for the more general class of monotone valuations.

The problem of \emph{fair division} is to find a partition $E=X_1\cup \cdots \cup X_n$ of the set of goods into bundles $X_i$ such that the \emph{allocation} of the bundles to the agents ($X_i$ goes to agent $i$) is \emph{fair} in some precisely defined way. Throughout, a partition into $n$ bundles means an ordered $n$-tuple of pairwise disjoint, possibly empty, sets whose union is the relevant ground set. Ideally, the allocation would be fair if no one \emph{envied} anyone else, that is, if for every pair of agents $i,j\in A$ the following inequality was satisfied:
\[
v_i(X_i)\geqslant v_i(X_j).
\]
This is unfortunately not possible in general when single goods cannot be divided. For a nonempty bundle $X$, let
\[
v^-_i(X)=v_i(X\setminus\{g\}),
\]
where $g$ is a good of maximum value $v_i(g)$ to agent $i$ among the goods in $X$. We also set $v^-_i(\emptyset)=0$. For additive valuations, we say that agent $i$ does not envy agent $j$ up to one good if $v_i(X_i)\geqslant v^-_i(X_j)$. The allocation is said to satisfy \emph{envy-freeness up to one good} (EF1, for short) if this condition holds for every pair of agents $i,j\in A$. For general monotone valuations, we use the existential version: if $X_j\neq \emptyset$, there must exist a good $g\in X_j$ such that $v_i(X_i)\geqslant v_i(X_j\setminus\{g\})$; if $X_j=\emptyset$, the condition is automatic.

It is known that for any instance of the fair division problem, an EF1 allocation exists \cite{Lipton2004}. However, in many real-world applications it is desirable that the bundles $X_i$ of the partition satisfy additional constraints, as mentioned above \cite{DrorFeldmanSegalHalevi2023,SuksompongSurvey2021}. It is therefore convenient to assume that the set of goods $E$ is equipped with some additional combinatorial structure.

In this paper we consider a setting in which parts of the partition are restricted to the family of \emph{independent} sets in a certain combinatorial structure, such as a graph, a matroid, or any simplicial complex. Recall that an \emph{independence structure}, or an \emph{abstract simplicial complex}, is any pair $M=(E,\mathcal{I})$, where $E$ is a finite set and $\mathcal{I}$ is a family of subsets of $E$ closed under taking subsets, that is, if $S\in \mathcal{I}$ and $X\subseteq S$, then $X\in \mathcal{I}$.

For a given independence structure $M$, we denote by $\chi(M)$ the \emph{chromatic number} of $M$, that is, the least number of parts in a partition of the ground set $E$ into independent sets. Clearly, we must have $\chi(M)\leqslant n=|A|$ in order to have at least a chance for a \emph{feasible} partition of the set of goods into $n$ independent bundles.

Given an independence structure $M$ and an integer $q\geqslant \chi(M)$, we say that $M$ is \emph{EF1-$q$-colorable} if for every nonnegative additive valuation on $E$ there exists an EF1 partition of $E$ into $q$ independent sets. A class of independence structures $\mathcal{M}$ has the \emph{EF1 property} if every $M\in \mathcal{M}$ is EF1-$q$-colorable for every integer $q\geqslant \chi(M)$.

\begin{problem}\label{Problem Main IS}
	Which classes of independence structures have the EF1 property?
\end{problem}

Our main motivation comes from the following intriguing conjecture (see \cite{SuksompongSurvey2021}).

\begin{conjecture}\label{Conjecture Matroid One Valuation}
	The class of matroids satisfies the EF1 property.
\end{conjecture}

Recall that a \emph{matroid} is an independence structure satisfying the \emph{exchange} property, that is, given two independent sets $X$ and $Y$ with $|X|>|Y|$, there is an element $x\in X\setminus Y$ such that $Y\cup\{x\}$ is independent. A prototypical matroid can be obtained by taking any finite set $E$ of vectors in an arbitrary vector space and defining the family of independent sets to consist of all linearly independent subsets of $E$. Another basic example is obtained from a graph $G$ by defining independent sets as those subsets of edges of $G$ which do not contain a cycle. Such structures appear, for instance, in many real-life problems arising in logistics, telecommunication, and financial networks. In particular, all feasible basic solutions in linear network flow problems, such as the Minimum Cost Flow Problem, the Maximum Flow Problem, or the Shortest Path Problem, are spanning forests or spanning trees \cite{Ahuja1993NetworkFT}.

Matroids are especially relevant here because they provide a broad language for modelling feasibility while still supporting structural arguments through exchange and restriction operations.

Conjecture \ref{Conjecture Matroid One Valuation} is known to hold for some wide classes of matroids, most notably for \emph{base-orderable} matroids and for \emph{regular} matroids; see \cite{BarmanBiswas2018,DrorFeldmanSegalHalevi2023}. Akrami, Raj, and V\'{e}gh \cite{Akrami2026matroids} proved that its validity follows from a conjecture of Gabow. One of our results points to another possible route, based on Hamiltonicity of matroids, which we describe below.

Assuming that the above conjecture is true, it is natural to ask for its stronger version, where the agents may have distinct additive valuations.

\begin{conjecture}\label{Conjecture Matroid Many Valuations}
	Let $M$ be a matroid on the set of goods $E$, with $\chi(M)\leqslant n$, and suppose that there are $n$ agents having additive valuations $v_1,\ldots, v_n$. Then there exists an EF1 partition of $E$ into $n$ independent sets of $M$.
\end{conjecture}

We make here a small step supporting the validity of Conjecture \ref{Conjecture Matroid Many Valuations} by proving its \emph{epistemic} version for hereditary classes of independence structures having the EF1 property.

Using a concept analogous to epistemic EFX \cite{CaragiannisEtAl2023EEFX}, we say that an allocation $(X_1,\ldots, X_n)$ is \emph{epistemic EF1} if for every $i\in A$ there exists (possibly another -- \emph{imaginary}) allocation $(Y_1,\ldots, Y_n)$ such that $X_i=Y_i$ and $v_i(X_i)\geqslant v^-_i(Y_j)$ for every $j\neq i$. For monotone valuations, the analogous definition uses the existential formulation: for every $j\neq i$ with $Y_j\neq \emptyset$, there must exist a good $g\in Y_j$ such that $v_i(X_i)\geqslant v_i(Y_j\setminus\{g\})$. Notice that these imaginary allocations may vary from one agent to another. In the constrained version of the problem we demand that all bundles in each of the imaginary allocations are independent sets of the given independence structure.

To state our main result we need to recall the following basic notion. Let $M=(E,\mathcal{I})$ be any independence structure and let $S$ be a subset of $E$. Then, the \emph{restriction} of $M$ to the set $S$ is the structure $M|_S=(S,\mathcal{I}_S)$ whose family of independent sets $\mathcal{I}_S$ consists of all sets of the form $X\cap S$, with $X\in \mathcal{I}$. We will refer to $M|_S$ as the \emph{substructure} of $M$ \emph{induced} by the set $S$. The class of independence structures $\mathcal{M}$ is called \emph{hereditary} if it is closed under induced substructures.

\begin{theorem}\label{Theorem Epistemic Matroid}
Let $\mathcal{M}$ be a hereditary class of independence structures having the EF1 property. Let $M=(E,\mathcal{I})$ be any structure from $\mathcal{M}$, with $\chi(M)\leqslant n$, and suppose that there are $n$ agents having arbitrary additive valuations $v_1,\ldots, v_n$. Then there exists an epistemic EF1 partition of $E$ into $n$ independent sets of $M$.
\end{theorem}

For the graphic matroid of a graph $G$, the parameter $\chi(M)$ is exactly the arboricity of $G$. Since graphic matroids are regular, the known EF1 property for regular matroids cited above implies the common-valuation forest-partition statement mentioned in the abstract whenever the number of agents is at least this arboricity.

The proof is similar to that of an analogous result by Akrami and Rathi \cite{AkramiRathi2025}, who established the existence of epistemic EFX partitions in the unconstrained case for arbitrary monotone valuations. The proof of Theorem \ref{Theorem Epistemic Matroid} is existential, but it has a useful algorithmic reading. This reading should not be understood as a second constructive proof; it merely makes explicit the finite choices used in the induction. It becomes an effective procedure whenever the common-valuation EF1 colorings assumed in the definition of the EF1 property can be found effectively and whenever membership in the auxiliary families $\mathcal{F}_i^k(S)$ can be tested, preferably with witnesses.

Our next result uses the following Hamiltonicity property of matroids. Let us call a matroid $M=(E,\mathcal{I})$ \emph{Hamiltonian} if there exists a cyclic ordering of the set $E$ such that every \emph{segment} of $r(M)$ consecutive elements forms a basis of $M$.

\begin{theorem}\label{Theorem Matroid Hamiltonian}
	Let $M=(E,\mathcal{I})$ be a Hamiltonian matroid with $\chi(M)\leqslant 2$. Suppose that two agents have the same monotone valuation $v$ defined on the subsets of $E$. Then there exists an EF1 partition of $E$ into two independent sets whose sizes differ by at most $1$.
\end{theorem}

A matroid $M=(E,\mathcal{I})$ with the rank function $r$ is called \emph{uniformly dense} if the following condition is satisfied for every nonempty subset $X\subseteq E$.
\begin{equation}
	\frac{|X|}{r(X)}\leqslant \frac{|E|}{r(E)}. 
\end{equation}
Kajitani, Ueno, and Miyano observed in \cite{KajitaniUenoMiyano1988} that any Hamiltonian matroid must be uniformly dense and made a conjecture that the opposite implication is also true.

\begin{conjecture}[\cite{KajitaniUenoMiyano1988}]\label{Conjecture Hamiltonian Density}
	A matroid $M$ is Hamiltonian if and only if it is uniformly dense.
\end{conjecture}

It is not hard to verify that any matroid $M$ whose ground set can be partitioned into $n$ disjoint bases is uniformly dense. In consequence, Theorem \ref{Theorem Matroid Hamiltonian} and Conjecture \ref{Conjecture Hamiltonian Density} imply the two-agent case of Conjecture \ref{Conjecture Matroid One Valuation} for matroids whose ground set can be partitioned into two bases. It is not clear whether the same holds in greater generality. It is known, however, that any uniformly dense matroid with $\gcd(r(M),|E|)=1$ is Hamiltonian, as proved by van de Heuvel and Thomass\'{e} \cite{vandeHeuvel2012}. In particular, any uniformly dense matroid of rank $r$ and $2r-1$ elements has an EF1 partition. Indeed, by the standard matroid coloring theorem, $\chi(M)=\left\lceil \max_{\emptyset\neq X\subseteq E}|X|/r(X)\right\rceil$. For a uniformly dense matroid of rank $r$ on $2r-1$ elements, this maximum is at most $(2r-1)/r<2$, so $\chi(M)\leqslant 2$.

\section{Proof of Theorem \ref{Theorem Epistemic Matroid}}
Let $\mathcal{M}$ be a hereditary class of independence structures having the EF1 property, let $M=(E,\mathcal{I})$ be a structure from $\mathcal{M}$ with $\chi(M)\leqslant n$, and let $v_i$ be the additive valuation of agent $i\in A$. We use the notation $v_i^-$ introduced in Section~2; thus, for a nonempty bundle $X\subseteq E$, we have
\[
v^-_i(X)=v_i(X\setminus\{g\}),
\]
where $g\in X$ is a good of maximum value for agent $i$ among the goods in $X$, and set $v^-_i(\emptyset)=0$. In other words, $v^-_i(X)$ is the minimum value of $v_i(X')$, where $X'$ is a subset of $X$ with $|X'|=|X|-1$ when $X\neq \emptyset$. Thus the inequality $v_i(X)\geqslant v^-_i(Y)$ means that agent $i$, while receiving bundle $X$, does not envy too much an agent receiving bundle $Y$.

For a subset of goods $S\subseteq E$ and an integer $k\geqslant 1$, every agent $i\in A$ defines a family of bundles $\mathcal{F}^k_i(S)$ as follows. For $k=1$, we put $X\in \mathcal{F}^1_i(S)$ if and only if $X=S$ and $X\in \mathcal{I}$. For $k\geqslant 2$, we put $X\in \mathcal{F}^k_i(S)$ if and only if $X\in \mathcal{I}$ and there is a partition $S=X\cup P_1\cup\cdots \cup P_{k-1}$, with $P_j\in \mathcal{I}$, such that
\[
v_i(X)\geqslant v^-_i(P_j)
\]
for all $j=1,2,\ldots,k-1$. The members of the family $\mathcal{F}^k_i(S)$ will be called \emph{$k$-epistemic} bundles in the subset $S$ for agent $i$. In the case $k=n$ and $S=E$, a partition $E=X_1\cup\cdots \cup X_n$ with $X_i\in \mathcal{F}^n_i(E)$ for all $i$ is exactly what we seek in Theorem \ref{Theorem Epistemic Matroid}.

Our argument can be sketched as follows. Starting from an EF1 partition for a suitable common valuation, we define a bipartite graph $G$ between vertices representing the parts of the partition $(X_1,\ldots,X_n)$ and the set of agents $A=\{1,\ldots,n\}$, joining $i$ to the vertex representing $X_j$ whenever $X_j\in \mathcal{F}^n_i(E)$. If $G$ has a perfect matching, we are done. Otherwise, Corollary \ref{Corollary Hall} yields a matched subfamily of bundles that is already suitable for some agents, and we apply induction to the remaining agents and the remaining goods. Finally, we glue the two resulting allocations into a full solution.

In order to realize this plan, we will need the following key lemma.

\begin{lemma}\label{Lemma k-epistemic}
	Let $1 \leqslant k \leqslant n-1$ and let $S\subseteq E$. Suppose that $X\in \mathcal{I}$, $S\cap X = \emptyset$, and $X\notin \mathcal{F}^{k+1}_i(S\cup X)$ for some $i\in A$. Then, for any $Y\in \mathcal{F}^k_i(S)$, we also have $Y\in \mathcal{F}^{k+1}_i(S\cup X)$. 
\end{lemma}
\begin{proof}
	Let $S=Y\cup P_1\cup \cdots \cup P_{k-1}$ be a partition witnessing that $Y\in \mathcal{F}^k_i(S)$. Then we have $v_i(Y)\geqslant v^-_i(P_j)$ for all $j=1,2,\ldots,k-1$. If $v_i(X)\geqslant v_i(Y)$, then $v_i(X)\geqslant v^-_i(P_j)$ for all $j=1,2,\ldots,k-1$. Since $v^-_i(Y)\leqslant v_i(Y)$, the inequality $v_i(X)\geqslant v_i(Y)$ also gives $v_i(X)\geqslant v^-_i(Y)$. Hence the partition
	\[
	S\cup X=X\cup Y\cup P_1\cup \cdots \cup P_{k-1}
	\]
	implies that $X\in \mathcal{F}^{k+1}_i(S\cup X)$, a contradiction. Therefore $v_i(X)<v_i(Y)$. Since valuations are nonnegative and additive, $v^-_i(X)\leqslant v_i(X)<v_i(Y)$. Together with the inequalities $v_i(Y)\geqslant v^-_i(P_j)$, the same partition
	\[
	S\cup X=X\cup Y\cup P_1\cup \cdots \cup P_{k-1}
	\]
	shows that $Y\in \mathcal{F}^{k+1}_i(S\cup X)$.
\end{proof}
This lemma can be easily generalized to the following form which allows for the above mentioned gluing of matchings.

\begin{lemma}\label{Lemma Gluing Matchings}
	Let $1 \leqslant k \leqslant n-1$ and let $E=S\cup T$ be a partition. Suppose that $T=X_{k+1}\cup \cdots \cup X_n$ is a partition of $T$ into subsets from $\mathcal{I}$ such that, for some $i\in A$, we have $X_j\notin \mathcal{F}^{n}_i(E)$ for all $j=k+1,\ldots,n$. Then, for any $Y\in \mathcal{F}^k_i(S)$, we also have $Y\in \mathcal{F}^n_i(E)$.
\end{lemma}
\begin{proof}
	Suppose that $Y\in \mathcal{F}^k_i(S)$. The reasoning is by induction on the number $n-k$ of parts in the partition of $T$. If $n-k=1$, then we get the assertion directly from Lemma \ref{Lemma k-epistemic}, by putting $X=X_n$. Assume now that $n-k\geqslant 2$ and the assertion holds when the number of parts in the partition of $T$ is strictly smaller than $n-k$.
	
	First notice that each part $X_j$ from the partition of $T$, with $k+1 \leqslant j \leqslant n-1$, must satisfy $X_j\notin \mathcal{F}^{n-1}_i(E\setminus X_n)$. Indeed, suppose on the contrary that $X_j\in \mathcal{F}^{n-1}_i(E\setminus X_n)$ and put $k=n-1$, $Y=X_j$, $X=X_n$, and $S=E\setminus X_n$ in Lemma \ref{Lemma k-epistemic}. Then, we would get that $X_j\in \mathcal{F}^n_i(E)$, contrary to the assumption. Therefore, we may delete the last part $X_n$ and consider the partition $E\setminus X_n=S\cup T'$, where $T'=X_{k+1}\cup \cdots \cup X_{n-1}$, which still satisfies the assumptions of the lemma.
	
	Now, by the inductive assumption we have $Y\in \mathcal{F}^{n-1}_i(E\setminus X_n)$. Finally, applying Lemma \ref{Lemma k-epistemic} once again, we get $Y\in \mathcal{F}^{n}_i(E)$. 
	
\end{proof}

This lemma allows for stating the key gluing property which allows for an inductive argument in the proof of Theorem \ref{Theorem Epistemic Matroid}.

\begin{corollary}\label{Corollary Gluing Partitions}
	Let $M=(E,\mathcal{I})$ be an independence structure and let $E=S\cup T$ be a partition such that $S=Y_1\cup \cdots \cup Y_k$ and $T=X_{k+1}\cup \cdots \cup X_n$, where $Y_i,X_j\in \mathcal{I}$. Moreover, suppose that $Y_i\in \mathcal{F}^{k}_i(S)$, $X_j\in \mathcal{F}^{n}_j(E)$, and $X_j\notin \mathcal{F}^{n}_i(E)$ for all $i=1,2,\ldots,k$ and $j=k+1,\ldots,n$. Then $(Y_1,\ldots,Y_k,X_{k+1},\ldots,X_n)$ is an epistemic EF1 allocation of independent sets of $M$.
\end{corollary}
\begin{proof}
	Applying Lemma \ref{Lemma Gluing Matchings} to each part $Y_i$ gives that $Y_i\in \mathcal{F}^{n}_i(E)$. Each part of the partition is therefore an epistemic bundle for the corresponding agent.
\end{proof}

We use Hall's theorem in the following form. Recall that for a bipartite graph with vertex classes $A$ and $B$, and for $U\subseteq B$, we denote by $N(U)\subseteq A$ the set of vertices of $A$ adjacent to at least one vertex of $U$.

\begin{lemma}\label{Lemma Hall's Theorem}
	
Let $G$ be a bipartite graph on two sets of vertices, $A=\{1,2,\ldots, n\}$ and $B=\{b_1,b_2,\ldots, b_n\}$. There exists a perfect matching in $G$ if and only if for each subset $U\subseteq B$,
$$
|N(U)|\geqslant |U|.
$$
\end{lemma}

For the proof of Theorem \ref{Theorem Epistemic Matroid} we need the following simple consequence of Hall's theorem.

\begin{corollary}\label{Corollary Hall}
Let $G$ be a non-empty bipartite graph on two sets of vertices, $A=\{1,2,\ldots,n\}$ and $B=\{b_1,b_2,\ldots,b_n\}$, such that some vertex $i_0\in A$ is adjacent to every vertex of $B$. Then there exists a non-empty subset $U\subseteq B$ such that $|N(U)|=|U|$. Moreover, there exists a perfect matching $H$ in the subgraph $G_U=G[U\cup N(U)]$ of $G$ induced by the vertex set $U\cup N(U)$.
\end{corollary}
\begin{proof}
	If $G$ has a perfect matching, then clearly $A=N(B)$ and we may take $U=B$. If there is no perfect matching in $G$, then there must be a subset $Q$ of $B$ violating Hall's condition, that is, $|N(Q)|\leqslant |Q|-1$. We may assume that $Q$ is minimal with respect to inclusion. Since $i_0\in N(Q)$, we have $Q\neq \emptyset$, and therefore $|Q|\geqslant 2$. Let $U$ be any non-empty subset of $Q$ with $|U|=|Q|-1$. By the minimality of $Q$, the set $U$ cannot violate Hall's condition, so
	\[
	|U|\leqslant |N(U)|\leqslant |N(Q)|\leqslant |Q|-1=|U|.
	\]
	It follows that $|N(U)|=|U|$.
	
	Since $Q$ is a minimal set violating Hall's condition in the graph $G$, every subset $T$ of $U$ must satisfy it (in $G$). Since $N(T)\subseteq N(U)$, Hall's condition for $T$ is also satisfied in the graph $G_U=G[U\cup N(U)]$. Thus, there must be a perfect matching in $G_U$.
\end{proof}

Now we are ready for the proof of Theorem \ref{Theorem Epistemic Matroid}.

\begin{proof}[Proof of Theorem \ref{Theorem Epistemic Matroid}]
	We argue by induction on the number of agents $n$. For $n=1$, the inequality $\chi(M)\leqslant 1$ implies that $E$ itself is independent, so the unique allocation $(E)$ is trivially epistemic EF1.
	
	Assume now that $n\geqslant 2$ and that the statement holds for all smaller numbers of agents. Since $\mathcal{M}$ has the EF1 property, and since the EF1 property is assumed for every additive valuation, we may apply it to the common valuation $v_n$. Thus there exists an EF1 partition $(X_1,\ldots,X_n)$ of $E$ into independent sets for the instance in which all agents have the common additive valuation $v_n$. Let $B=\{b_1,\ldots,b_n\}$, where $b_j$ is a formal vertex representing the bundle $X_j$. Let $G$ be the bipartite graph with bipartition classes $A=\{1,\ldots,n\}$ and $B$, in which $b_j$ is joined by an edge to $i$ whenever $X_j$ is an $n$-epistemic bundle for agent $i$ in the whole set $E$, that is, when $X_j\in \mathcal{F}^n_i(E)$.
	
	For every pair of bundles $X_a,X_b$, the initial allocation is EF1 with respect to the common valuation $v_n$, and therefore
	\[
	v_n(X_a)\geqslant v^-_n(X_b).
	\]
	Hence, for each $a$, the same partition $(X_1,\ldots,X_n)$ witnesses that $X_a\in \mathcal{F}^n_n(E)$. Consequently, agent $n$ is adjacent to every vertex of $B$.
	
	Let $U\subseteq B$ be a set satisfying the properties guaranteed by Corollary \ref{Corollary Hall}, that is, $|U|=|N(U)|$ and there is a perfect matching $F$ between $U$ and $N(U)$. If $U=B$, then $F$ is a perfect matching of the whole graph $G$. In that case each agent is matched to an $n$-epistemic bundle, and the resulting allocation is epistemic EF1. Thus we may assume that $U$ is a non-empty proper subset of $B$.
	
	Since agent $n$ is adjacent to every vertex of $B$, we have $n\in N(U)$. After possibly renumbering the bundles and the agents other than $n$, we may write $U$ as a set of bundle vertices:
	\[
	U=\{b_{k+1},\ldots,b_n\}\qquad\text{and}\qquad N(U)=\{k+1,\ldots,n\},
	\]
	where $1\leqslant k\leqslant n-1$, and the matching $F$ consists of the edges $(j,b_j)$ for $j=k+1,\ldots,n$. Let us additionally denote $A_1=\{1,2,\ldots,k\}$ and $A_2=\{k+1,\ldots,n\}$.
	
	By construction, every agent in $A_2$ is matched to a vertex representing a bundle from $\mathcal{F}^n_j(E)$. Moreover, any agent $i\in A_1$ is not adjacent to any vertex in $U$, since $i\notin N(U)$. In other words, we have $X_j\notin \mathcal{F}^n_i(E)$ for every $i\in A_1$ and every $j\in A_2$.
	
	Now let
	\[
	S=X_1\cup \cdots \cup X_k\qquad\text{and}\qquad T=X_{k+1}\cup\cdots \cup X_n.
	\]
	Consider the restriction $M|_S$ of $M$ to the set $S$. Since $X_1,\ldots,X_k$ are independent and partition $S$, we have $\chi(M|_S)\leqslant k$. Because $\mathcal{M}$ is hereditary, we also have $M|_S\in \mathcal{M}$. Therefore, the induction hypothesis applies to $M|_S$ and the agents in $A_1$, yielding a partition
	\[
	S=Y_1\cup\cdots \cup Y_k
	\]
	that is epistemic EF1 within the structure $M|_S$. Equivalently, $Y_i\in \mathcal{F}^k_i(S)$ for every $i=1,\ldots,k$. Since all parts $Y_i$ are independent in $M|_S$, they are also independent in $M$.
	
	Finally, Corollary \ref{Corollary Gluing Partitions} applied to the partition
	\[
	E=Y_1\cup\cdots \cup Y_k\cup X_{k+1}\cup\cdots \cup X_n
	\]
	shows that this allocation is epistemic EF1 in $M$.
\end{proof}

\noindent\textbf{Algorithmic reading of Theorem~\ref{Theorem Epistemic Matroid}.}
We present the recursive construction implicit in the proof. The purpose of the pseudocode is only to identify the choices made by the inductive argument. It uses the following two black boxes.

\begin{itemize}
    \item \textsc{EF1Color}$(M|_S,q,w)$ returns an EF1 partition of $S$ into $q$ independent sets of the induced structure $M|_S$, for the common additive valuation $w$, whenever $M|_S\in\mathcal{M}$ and $\chi(M|_S)\leqslant q$.
    \item \textsc{IsEpiBundle}$(i,X,S,q)$ decides whether $X\in\mathcal{F}_i^q(S)$ and, when the answer is positive, may return a witnessing partition of $S\setminus X$.
\end{itemize}

Thus the pseudocode is constructive exactly to the extent that these two tasks are constructive. In particular, no additional existence assertion is hidden in the algorithm beyond the EF1-colorability assumption and the membership tests for the families $\mathcal{F}_i^q(S)$. We write \textsc{BuildEpiEF1Partition} for the recursive routine below. The name is only mnemonic: it records the inductive construction used in the proof, and is not meant as a separate algorithmic theorem.

\begin{algorithm}[H]
\SetAlgoLined
\SetKwFunction{EFColor}{EF1Color}
\SetKwFunction{IsEpiBundle}{IsEpiBundle}
\SetKwFunction{Alg}{BuildEpiEF1Partition}
\KwIn{An induced substructure $M|_S=(S,\mathcal I_S)\in\mathcal{M}$ with $\chi(M|_S)\leqslant q$, a non-empty set $A'\subseteq A$ of $q$ agents, and the additive valuations $(v_i)_{i\in A'}$.}
\KwOut{A partition $(Z_i)_{i\in A'}$ of $S$ into independent sets of $M|_S$ such that $Z_i\in\mathcal{F}_i^q(S)$ for every $i\in A'$.}
\If{$q=1$}{
    let $i$ be the unique agent in $A'$, set $Z_i\leftarrow S$, and \Return $(Z_i)_{i\in A'}$\;
}
Choose an agent $p\in A'$\;
$(X_1,\ldots,X_q)\leftarrow$ \EFColor{$M|_S,q,v_p$}\tcp*[r]{common valuation $v_p$}
Let $B=\{b_1,\ldots,b_q\}$ be vertices representing $X_1,\ldots,X_q$\;
Construct the bipartite graph $G$ with left side $A'$ and right side $B$\;
\ForEach{$i\in A'$}{
    \For{$j=1$ \KwTo $q$}{
        \If{\IsEpiBundle{$i,X_j,S,q$}}{
            add the edge $(i,b_j)$ to $G$\;
        }
    }
}
\If{$G$ has a perfect matching $F$}{
    set $Z_i\leftarrow X_j$ whenever $(i,b_j)\in F$\;
    \Return $(Z_i)_{i\in A'}$\;
}
Find an inclusion-minimal Hall-deficient set $Q\subseteq B$, so $|N(Q)|\leqslant |Q|-1$\;
Choose $b^*\in Q$ and set $U\leftarrow Q\setminus\{b^*\}$\tcp*[r]{$U\neq\emptyset$ and $|N(U)|=|U|$}
Find a perfect matching $F$ in the induced graph $G[U\cup N(U)]$\;
Set $A_2\leftarrow N(U)$ and $A_1\leftarrow A'\setminus A_2$\;
Set $S_1\leftarrow\displaystyle\bigcup_{b_j\in B\setminus U}X_j$\;
For every $i\in A_2$, set $Z_i\leftarrow X_j$ whenever $(i,b_j)\in F$\;
$(Z_i)_{i\in A_1}\leftarrow$ \Alg{$M|_{S_1},A_1,(v_i)_{i\in A_1}$}\;
\Return $(Z_i)_{i\in A'}$\;
\caption{\textsc{BuildEpiEF1Partition}: recursive reading of the proof of Theorem~\ref{Theorem Epistemic Matroid}}\label{alg:epistemic-ef1}
\end{algorithm}

The top-level call is \textsc{BuildEpiEF1Partition}$(M,A,(v_i)_{i\in A})$. In a call with $q$ agents, the non-oracle work consists of $q^2$ membership tests for the relation $X\in\mathcal{F}_i^q(S)$ and standard bipartite matching or Hall-set computations. Since the number of agents strictly decreases in every recursive call, there are at most $n$ calls. Hence, if \textsc{EF1Color} and \textsc{IsEpiBundle} are polynomial-time routines for the class under consideration, then the recursive construction is polynomial-time as well.

\noindent\textbf{EXAMPLE (graphic matroid of $K_6$).}
Let the vertex set of $K_6$ be $\{1,2,3,4,5,6\}$ and let the goods be its edges:
\[
E = \{(i,j) : 1 \leqslant i < j \leqslant 6\}.
\]

We consider three agents $A'=\{1,2,3\}$ with strictly positive additive valuations:
\begin{align*}
v_1(e)&=1 \quad \text{for every } e\in E,\\
v_2(e)&=
\begin{cases}
3 & \text{if } e = (i,i+1) : 1 \leqslant i \leqslant 5,\\
1 & \text{otherwise},
\end{cases}\\
v_3(e)&=
\begin{cases}
5 & \text{if } e = (i,i+1) : 1 \leqslant i \leqslant 5,\\
1 & \text{otherwise}.
\end{cases}
\end{align*}

\medskip
\noindent
\textbf{EF1 partition for $v_1$.}
Since $v_1$ is uniform, we partition $E$ into three edge-disjoint spanning trees:
\begin{align*}
X_1 &= \{(1,2),(2,3),(3,4),(4,5),(5,6)\},\\
X_2 &= \{(1,3),(3,5),(2,5),(2,4),(4,6)\},\\
X_3 &= \{(1,4),(1,5),(1,6),(2,6),(3,6)\}.
\end{align*}
Each $X_j$ is a tree, hence independent, and $(X_1,X_2,X_3)$ is EF1 for agent $1$.

\medskip
\noindent
\textbf{Values of bundles.} We have:
\begin{align*}
v_1(X_1)&=1+1+1+1+1=5,\\
v_1(X_2)&=1+1+1+1+1=5,\\
v_1(X_3)&=1+1+1+1+1=5,
\end{align*}
\begin{align*}
v_2(X_1)&=3+3+3+3+3=15,\\
v_2(X_2)&=1+1+1+1+1=5,\\
v_2(X_3)&=1+1+1+1+1=5,
\end{align*}
\begin{align*}
v_3(X_1)&=5+5+5+5+5=25,\\
v_3(X_2)&=1+1+1+1+1=5,\\
v_3(X_3)&=1+1+1+1+1=5.
\end{align*}

\medskip
\noindent
\textbf{Acceptable bundles.} Every bundle is acceptable for agent 1 (even without mixing the two other bundles), i.e., $X_1,X_2,X_3\in \mathcal{F}_1^{3}(E)$. On the other hand, only $X_1$ is acceptable for agent $2$. Without loss of generality, consider the set $X_2$. In any partition of $X_1\cup X_3$ into two independent sets $X_1^\prime$ and $X_3^\prime$ we must have $|X_1^\prime|=|X_3^\prime|=5$ and for one of those two sets $X^\prime\in \{X_1^\prime,X_3^\prime\}$, we must have $|X_1\cap X^\prime| \geq 3$, which implies $v_2(X^\prime)\geqslant 3+3+3+1+1 = 11$ and $v_2^-(X^\prime)\geqslant 3+3+1+1 = 8 > v_2(X_2)$, so no EF1-partition exists. For the same reason the only acceptable bundle for agent 3 is $X_1$ and we have $X_1\in \mathcal{F}_2^{3}(E)$, $X_2, X_3\notin \mathcal{F}_2^{3}(E)$, $X_1\in \mathcal{F}_3^{3}(E)$, $X_2, X_3\notin \mathcal{F}_3^{3}(E)$. Consequently, we obtain the bipartite graph with edges
$$
(1,X_1),(1,X_2),(1,X_3),\quad (2,X_1),\quad (3,X_1).
$$
Obviously it does not have a perfect matching. An inclusion-minimal Hall-deficient set is $Q=\{X_2,X_3\}$, as $|N(Q)|=|\{1\}|=1$.

\medskip
\noindent
\textbf{Recursive step.}
Let the subset of $Q$ be $U=\{X_2\}$. We have $|U|=|Q|-1=1$ and $|N(U)|=|\{1\}|=1$. We fix the matching $(1,X_2)$. The remaining goods form the set
$$
E_1 = E\setminus X_2 = X_1 \cup X_3.
$$
We will now recursively solve the problem for agents 2 and 3 on $E_1$. First, we find an EF1 allocation with respect to the valuation $v_2$. An example is the following partition:
\begin{align*}
Y_1 &= \{(1,2),(2,3),(3,4),(1,5),(1,6)\},\\
Y_2 &= \{(1,4),(2,6),(3,6),(4,5),(5,6)\}.
\end{align*}

\medskip
\noindent
\textbf{Values of bundles, acceptable bundles.} We have:
\begin{align*}
v_2(Y_1)&=3+3+3+1+1=11\geqslant v_2^-(Y_2)=1+1+1+3=6,\\
v_2(Y_2)&=1+1+1+3+3=9 \geqslant v_2^-(Y_1)=3+3+1+1=8,
\end{align*}
\begin{align*}
v_3(Y_1)&=5+5+5+1+1=17 \geqslant v_3^-(Y_2)=1+1+1+5=8,\\
v_3(Y_2)&=1+1+1+5+5=13 \geqslant v_3^-(Y_1)=5+5+1+1=12.
\end{align*}
This means that $Y_1,Y_2\in \mathcal{F}_2^{2}(E_1)\cap\mathcal{F}_3^{2}(E_1)$, the corresponding bipartite graph is 
$$
(2,Y_1),(2,Y_2),\quad (3,Y_1),(3,Y_2),
$$
and a sample perfect matching:
$$
(2,Y_1)\quad (3,Y_2).
$$
This gives the final EF1-epistemic assignment:
$$
(1,X_2)\quad (2,Y_1)\quad (3,Y_2).
$$

\section{Proof of Theorem \ref{Theorem Matroid Hamiltonian}}

We start with the following key lemma.

\begin{lemma}\label{Lemma Circle Equal}
	Let the set $E$ of goods be arranged in a cyclic order. Suppose that $v$ is a monotone valuation of these goods. Then there is an EF1 partition of $E$ into two cyclic intervals whose sizes differ by at most one.
\end{lemma}
\begin{proof}
	Assume first that the size of the set $E$ is even. Let $(e_1,e_2,\ldots, e_{2k})$ be a fixed cyclic ordering of the set $E$ of goods. Let
	\[
	A_i=\{e_i,e_{i+1}, \ldots, e_{i+k-1}\},
	\]
	for $i=1,2,\ldots, 2k$, be a sequence of cyclic intervals of size $k$, where the indices of elements $e_j$ are considered modulo $2k$. Clearly, this leads to a sequence $P_i=(A_i,B_i)$ of partitions, $E=A_i\cup B_i$, where $B_i=A_{i+k}$, for each $i=1,2,\ldots, 2k$.
	
	Each pair $P_i$ is either \emph{equitable}, \emph{increasing}, or \emph{decreasing}, depending on whether $v(A_i)=v(B_i)$, $v(A_i)<v(B_i)$, or $v(A_i)>v(B_i)$, respectively. If there is any equitable pair, then we are done. Otherwise, each pair is either increasing or decreasing, but not all pairs can be of the same type. Indeed, since $A_{i+k}=B_i$ for all $i=1,2,\ldots,2k$, every two pairs $P_i$ and $P_{i+k}$ must have opposite types.
	
	Let $j$ be an index for which the pair $P_j$ is increasing while the pair $P_{j+1}$ is decreasing. We will demonstrate that one of these pairs is an EF1 partition of $E$.
	
	By shifting cyclically the indices, we may assume that $j=1$. Then we have $v(A_1)<v(B_1)$ and $v(A_2)>v(B_2)$. Let us denote $X=A_1\setminus\{e_1\}$ and $Y=B_1\setminus\{e_{k+1}\}$. In consequence, we have $A_2=X\cup\{e_{k+1}\}$ and $B_2=Y\cup\{e_1\}$. We need to consider two cases.
	\begin{case}
		$v(X)\geqslant v(Y)$.
	\end{case}
Then the pair $P_1=(A_1,B_1)$ satisfies EF1. Indeed, on one hand, we have $v(A_1)<v(B_1)$, by assumption. Thus, the agent receiving $B_1$ does not envy $A_1$, and by monotonicity, the required inequality remains true after removing a good from $A_1$. On the other hand,
\[
v(A_1)\geqslant v(X)\geqslant v(Y)=v(B_1\setminus\{e_{k+1}\}).
\]
\begin{case}
	$v(X)<v(Y)$.
\end{case}
Then the pair $P_2=(A_2,B_2)$ satisfies EF1. Indeed, now we have $v(A_2)>v(B_2)$, by assumption. Thus, the agent receiving $A_2$ does not envy $B_2$, and by monotonicity, the required inequality remains true after removing a good from $B_2$. On the other hand,
\[
v(B_2)\geqslant v(Y)>v(X)=v(A_2\setminus\{e_{k+1}\}).
\]
This completes the proof of the lemma in the even case.

For the odd case, let $|E|=2k-1$. Add a new dummy good $d\notin E$, insert it anywhere into the cyclic order, and denote the resulting cyclically ordered set by $E'=E\cup\{d\}$. Define a monotone valuation $v'\colon 2^{E'}\to \mathbb{R}_{\geqslant 0}$ by
\[
v'(S)=v(S\setminus\{d\})\qquad\text{for every }S\subseteq E'.
\]
Then $d$ is neutral. By the even case, the set $E'$ admits an EF1 partition into two cyclic intervals $A$ and $B$ of equal size $k$. Assume that $d\in A$ and put $A'=A\setminus\{d\}$. Then $A'$ and $B$ form a partition of the original set $E$ into two cyclic intervals whose sizes differ by at most one.

It remains to verify that the resulting partition of the original set $E$ is EF1. Consider the two directed EF1 comparisons in the partition $(A,B)$ of $E'$. If the witness good for a given comparison is different from $d$, then the same witness works after deleting $d$. If the witness good is $d$, then the envied bundle is $A$, since we assumed that $d\in A$; after replacing $A$ by $A'$, its value with respect to $v'$ does not increase, and the required inequality remains valid. Thus $(A',B)$ is EF1 for the original valuation $v$. This also covers the case $|E|=1$, where one of the two intervals may become empty; the EF1 condition is then interpreted using the convention for empty bundles stated above.
\end{proof}

\begin{proof}[Proof of Theorem \ref{Theorem Matroid Hamiltonian}]
Let $M=(E,\mathcal{I})$ be a Hamiltonian matroid of rank $r$, with $\chi(M)\leqslant 2$, and let the two agents have a common monotone valuation $v$. Since $\chi(M)\leqslant 2$, the set $E$ can be partitioned into two independent sets, and therefore $|E|\leqslant 2r$. By Hamiltonicity of $M$, the set $E$ can be arranged on a cycle so that each cyclic interval of length $r$ is a basis of $M$. By Lemma \ref{Lemma Circle Equal}, $E$ has an EF1 partition $E=A\cup B$ such that $A$ and $B$ are cyclic intervals with $\bigl||A|-|B|\bigr|\leqslant 1$. Each of the two intervals has length at most $r$. Every cyclic interval of length at most $r$ is contained in a cyclic interval of length $r$, which is a basis by Hamiltonicity; therefore both parts are independent.
\end{proof}

\noindent\textbf{Algorithmic reading of Theorem~\ref{Theorem Matroid Hamiltonian}.}
The procedure below does not attempt to recognize Hamiltonicity. Instead, it records the finite choices made by the proof once a Hamiltonian cyclic ordering is supplied as a certificate. The subroutine \textsc{CyclicEF1Cut} is exactly the search contained in Lemma~\ref{Lemma Circle Equal}; Hamiltonicity is used only after this search, to certify that the two cyclic intervals it returns are independent.

\begingroup
\SetAlgorithmName{Subroutine}{subroutine}{List of Subroutines}
\begin{algorithm}[H]
\SetAlgoLined
\SetKwFunction{CyclicCut}{CyclicEF1Cut}
\renewcommand{\thealgocf}{}
\KwIn{A cyclic ordering $C=(e_1,\ldots,e_m)$ of a finite set $E$, and a common monotone valuation $v$.}
\KwOut{A partition $(A,B)$ of $E$ into two cyclic intervals such that $(A,B)$ is EF1 and $\bigl||A|-|B|\bigr|\leqslant 1$.}
\If{$m=0$}{
    \Return $(\emptyset,\emptyset)$\;
}
\If{$m$ is odd}{
    Add a dummy good $d\notin E$ with $v'(T)=v(T\setminus\{d\})$, and place it after $e_m$ in the cycle\;
    $(A',B')\leftarrow$ \CyclicCut{$(e_1,\ldots,e_m,d),v'$}\tcp*[r]{now the ground set has even size}
    \Return $(A'\setminus\{d\},B'\setminus\{d\})$\;
}
Set $k\leftarrow m/2$\;
\For{$t=1$ \KwTo $m$}{
    $A_t\leftarrow\{e_t,e_{t+1},\ldots,e_{t+k-1}\}$ and $B_t\leftarrow E\setminus A_t$\tcp*[r]{indices are cyclic}
    \If{$v(A_t)=v(B_t)$}{
        \Return $(A_t,B_t)$\;
    }
}
Choose $t$ such that $v(A_t)<v(B_t)$ and $v(A_{t+1})>v(B_{t+1})$\tcp*[r]{such a sign change exists on the cycle}
Set $X\leftarrow A_t\setminus\{e_t\}$ and $Y\leftarrow B_t\setminus\{e_{t+k}\}$\;
\If{$v(X)\geqslant v(Y)$}{
    \Return $(A_t,B_t)$\;
}
\Else{
    \Return $(A_{t+1},B_{t+1})$\;
}
\caption{\textsc{CyclicEF1Cut}: the cyclic-interval search used in Lemma~\ref{Lemma Circle Equal}}\label{alg:cyclic-ef1-cut}
\end{algorithm}
\endgroup
\addtocounter{algocf}{-1}


\begin{algorithm}[H]
\SetAlgoLined
\SetKwFunction{CyclicCut}{CyclicEF1Cut}
\KwIn{A Hamiltonian matroid $M=(E,\mathcal I)$ of rank $r$ with $\chi(M)\leqslant 2$, a Hamiltonian cyclic ordering $C$ of $E$, and a common monotone valuation $v$.}
\KwOut{An EF1 partition $(A,B)$ of $E$ into two independent sets, with $\bigl||A|-|B|\bigr|\leqslant 1$.}
$(A,B)\leftarrow$ \CyclicCut{$C,v$}\;
\tcp{Since $\chi(M)\leqslant 2$, we have $|E|\leqslant 2r$; hence both returned intervals have size at most $r$.}
\tcp{Each interval of size at most $r$ is contained in a Hamiltonian interval of length $r$, hence is independent.}
\Return $(A,B)$\;
\caption{Algorithmic reading of Theorem~\ref{Theorem Matroid Hamiltonian}}\label{alg:hamiltonian-two-agent}
\end{algorithm}

\medskip
\noindent\textbf{EXAMPLE (graphic matroid of $K_4$).}
Let $G=K_4$ with vertex set $\{1,2,3,4\}$ and edge set
$$
E=\{12,13,14,23,24,34\}.
$$
Consider the graphic matroid $M=M(G)$, whose independent sets are forests in $G$. 
Its rank is $r=3$, and its bases are exactly the spanning trees of $K_4$. We claim that $M$ is Hamiltonian. To see that, consider the cyclic ordering:
$$
(13,24,14,23, 12, 34).
$$

Now assume that the valuation $v:2^E \to \mathbb{R}_{\ge 0}$ of both agents is defined as follows:
\begin{align*}
v(S) &= (\text{number of vertices covered by } S)\\
&+ 2\cdot (\text{number of edges in $S$ having the form $(i,i+1)$ }).
\end{align*}

This valuation is monotone (adding an edge never decreases neither the number of covered vertices nor the number of edges of a certain form), and is not additive, as e.g. $v(\{12\})+v(\{23\})=(2+2\cdot 1)+(2+2\cdot 1)=8$, but $v(\{12,23\})=3+2\cdot 2 = 7$.

\medskip
\noindent
\textbf{Cyclic intervals with sign change.}
Let $A_i$ be the cyclic intervals of length $3$, and $B_i=E\setminus A_i$.
We have:

\begin{align*}
A_1&=\{13,24,14\}, &v(A_1)&=4+2\cdot 0=4, &B_1&=\{23,12,34\}, &v(B_1)&=4+2\cdot 3=10,\\
A_2&=\{24,14,23\}, &v(A_2)&=4+2\cdot 1=6, &B_2&=\{12,34,13\}, &v(B_2)&=4+2\cdot 2=8, \\
A_3&=\{14,23,12\}, &v(A_3)&=4+2\cdot 2=8, &B_3&=\{34,13,24\}, &v(B_3)&=4+2\cdot 1=6.
\end{align*}

We observe that $v(A_2)<v(B_2)$ and $v(A_3)>v(B_3)$, i.e., $(A_2,B_2)$ is increasing, but $(A_3,B_3)$ is decreasing. 

\medskip
\noindent
\textbf{Application of Lemma \ref{Lemma Circle Equal}.}
We have 
\begin{align*}
X &= A_2 \setminus \{24\} = \{14,23\}, &v(X) &= 4+2\cdot 1=6,\\
Y &= B_2 \setminus \{12\} = \{34,13\}, &v(Y) &= 3+2\cdot 1=5.
\end{align*}
This means that $v(X)\geqslant v(Y)$ and Case 1 applies, i.e., the pair $(A_2,B_2)$ is the desired partition.

\medskip
\medskip

Let us present yet another simple example, simultaneously showing the aspects of Theorems \ref{Theorem Epistemic Matroid} and \ref{Theorem Matroid Hamiltonian}.

\medskip
\noindent\textbf{EXAMPLE (illustration of both theorems).}
The following example illustrates both Theorem~\ref{Theorem Epistemic Matroid} and Theorem~\ref{Theorem Matroid Hamiltonian}.  Let $G$ be the subgraph of the clique $K_6$ on vertex set $\{1,2,3,4,5,6\}$ whose seven edges are
\[
\begin{aligned}
 e_1&=12, & e_2&=13, & e_3&=14, &
 e_4&=25, & e_5&=26, & e_6&=35, & e_7&=46 .
\end{aligned}
\]
Let $M=M(G)$ be the graphic matroid. Thus the goods are the edges $e_1,\ldots,e_7$, and a bundle is feasible precisely when it is a forest. Since $G$ is connected on six vertices, $r(M)=5$; equivalently, the bases of $M$ are the spanning trees of $G$, each having five edges.

\noindent{}\textbf{Illustration of Theorem\ref{Theorem Epistemic Matroid}.} First consider three agents with additive valuations
\[
\begin{array}{c|ccccccc}
 & e_1 & e_2 & e_3 & e_4 & e_5 & e_6 & e_7\\
\hline
v_1 & 3 & 3 & 3 & 1 & 1 & 1 & 1\\
v_2 & 1 & 1 & 1 & 4 & 4 & 1 & 1\\
v_3 & 1 & 1 & 1 & 1 & 1 & 1 & 1
\end{array}
\]
and the partition
\[
X_1=\{e_1,e_2,e_3\},\qquad
X_2=\{e_4,e_5\},\qquad
X_3=\{e_6,e_7\}.
\]
Each $X_i$ is a forest, so this is a partition into independent sets. Moreover,
\[
v_1(X_1)=9,
\qquad
v_2(X_2)=8,
\qquad
v_3(X_3)=2.
\]
For agent $1$, the other two bundles have value $2$ each. For agent $2$, the other two bundles have values $3$ and $2$. For agent $3$, the values of the three bundles are $3,2,2$, and deleting any one edge from $X_1$ leaves value $2$. Hence, this allocation is already EF1, and therefore also epistemic EF1 for all agents, in particular for agent 3. In the recursive reading of Theorem~\ref{Theorem Epistemic Matroid}, this is the case when the perfect matching exists: if one starts with the valuation of agent $3$ and \textsc{EF1Color} returns the above partition for the unit valuation, then the corresponding bipartite graph contains the matching
$$
(a_1,X_1),\quad (a_2,X_2),\quad (a_3,X_3),
$$
so no recursive split is needed.

\noindent{}\textbf{Illustration of Theorem\ref{Theorem Matroid Hamiltonian}} The same graphic matroid also illustrates Theorem~\ref{Theorem Matroid Hamiltonian}. Consider the cyclic order
$$
(e_1,e_2,e_3,e_4,e_5,e_6,e_7).
$$
Every block of five consecutive edges in this order is a spanning tree of $G$, so this cyclic order is Hamiltonian for the graphic matroid. Let the common valuation (monotone but not additive) be defined by
\begin{align*}
v(S) &= (\text{number of edges in }  S \text{ incident with $3$})^2\\
& + (\text{number of edges in }  S \text{ incident with $4$})^2.
\end{align*}
The partition 
\begin{align*}
A=\{e_1,e_2,e_3\},\\
B=\{e_4,e_5,e_6,e_7\}
\end{align*}
satisfies the conditions, since both parts are forests, their sizes are $3$ and $4$, and $v(A)=v(B)=2$.

\section{Final remarks}\label{section:finalremarks}
We conclude the paper with some problems for future consideration. One direction is to look for the \emph{chores} variant of fair division. In this setting, agents prefer to have less valuable bundles since the envy occurs when the other agent's bundle has strictly smaller value. More formally, the allocation $(X_1,\ldots, X_n)$ is \emph{EF1 for chores} if for every pair $i\neq j$, either $X_i=\emptyset$, or there exists a chore $c\in X_i$ such that $v_i(X_i\setminus\{c\})\leqslant v_i(X_j)$. Consequently, the allocation $(X_1,\ldots, X_n)$ is \emph{epistemic EF1 for chores} if for every $i\in A$, there is another (\emph{imaginary}) allocation $(Y_1,\ldots, Y_n)$ such that $X_i=Y_i$ and for every $j\neq i$, either $X_i=\emptyset$, or there exists a chore $c\in X_i$ such that $v_i(X_i\setminus\{c\})\leqslant v_i(Y_j)$.

It is not hard to verify with this chore version of EF1 that the analogue of Theorem \ref{Theorem Epistemic Matroid} holds for chores. Clearly, the statement of Conjecture \ref{Conjecture Matroid One Valuation} is equivalent for chores and goods, since all agents have the same valuation. However, it is not clear if the same holds for the statement of Conjecture \ref{Conjecture Matroid Many Valuations}.

\begin{problem}
	Assuming validity of Conjecture \ref{Conjecture Matroid Many Valuations}, is the same statement true for chores?
\end{problem}

Another interesting extension is to allow the agents not only to have arbitrary valuations but also to have individual constraints. Suppose that $M_1,M_2,\ldots, M_n$ are matroids on the same ground set $E$, each satisfying the inequality $\chi(M_i)\leqslant n$. Is it then true that there is an EF1 partition of $E$ such that the $i$-th agent's bundle is an independent set in the matroid $M_i$? Notice that it is not a priori obvious that any partition into such sets of $E$ always exists, however, it was confirmed independently several times (see \cite{DrorFeldmanSegalHalevi2023}). Below we formulate a general conjecture.

\begin{conjecture}\label{Conjecture Various Matroids}
Let $M_1, \ldots, M_n$ be a collection of matroids on the same ground set $E$ satisfying $\chi(M_i)\leqslant n$, for all $i=1,2,\ldots, n$. Suppose that there are $n$ agents with additive valuations $v_1,\ldots, v_n$ of the set of goods (or chores) $E$. Then there exists an EF1 partition $E=X_1\cup \cdots \cup X_n$ such that each set $X_i$ is independent in the matroid $M_i$, for all $i=1,2,\ldots,n$.
\end{conjecture}

Again, it would be nice to know if this statement follows from Conjecture \ref{Conjecture Matroid Many Valuations}.

\begin{problem}
	Does Conjecture \ref{Conjecture Various Matroids} follow from Conjecture \ref{Conjecture Matroid Many Valuations}?
\end{problem}

Finally, it would be nice to recognize some other independence structures beyond matroids with a chance for having EF1 partitions, at least in the case of one additive valuation and two agents. One natural candidate arises from a classic combinatorial problem of \emph{hypergraphs} with \emph{property B}.

Suppose that $H$ is a \emph{$k$-uniform} hypergraph on the vertex set $V$. A subset $X$ of the set $V$ is called \emph{independent} if there is no hyperedge of $H$ that is fully contained in $X$. A hypergraph $H$ is called \emph{bipartite} if there is a partition of $V$ into two independent sets. In other words, $H$ is a \emph{$2$-colorable} hypergraph, which is denoted by $\chi(H)\leqslant 2$. Suppose now that there is an additive valuation $v$ defined on the set of goods (or chores) $V$. When is it possible to find a $2$-coloring of $H$ with color classes forming an EF1 partition of $V$?

One easy necessary condition follows from taking $v$ as a unary function assigning value one to every item. Then the two color classes $A,B$ of the bipartition $V=A\cup B$ should be almost equal, that is, should satisfy the inequality $||A|-|B||\leqslant 1$. Let us call bipartite hypergraphs having such bipartition \emph{balanced}.

\begin{problem}
	Characterize bipartite balanced $k$-uniform hypergraphs having an EF1 bipartition, for any additive valuation $v$ and two agents.
\end{problem}

\bibliographystyle{abbrv}
\bibliography{fairallbib}
\end{document}